\documentclass[12pt]{amsart}
\usepackage[utf8]{inputenc}
\usepackage{amsmath,amsfonts}
\usepackage{amsthm}
\usepackage{amssymb}
\usepackage{hyperref}
\usepackage[shortlabels]{enumitem}
\usepackage{bm}
\usepackage{multirow}
\usepackage{tikz-cd}
\newtheorem{thm}{Theorem}[section]
\newtheorem{lem}[thm]{Lemma}
\newtheorem{prop}[thm]{Proposition}
\newtheorem{cor}[thm]{Corollary}
\newtheorem*{thm*}{Theorem}

\theoremstyle{definition}
\newtheorem{dfn}[thm]{Definition}
\newtheorem*{dfn*}{Definition}

\newtheorem{rem}[thm]{Remark}
\newtheorem{ques}{Question}
\newtheorem*{ques*}{Question}
\newcommand{\ra}{\rightarrow}

\newcommand{\N}{\mathbb{N}}

\newcommand{\C}{\mathbb{C}}
\newcommand{\Z}{\mathbb{Z}}

\title{Embeddability between $3$-manifolds is not stable}
\author{Giulio Belletti}
\address{IRMP, UC Louvain, Chemin du Cyclotron 2, bte L7.01.02, 1348 Louvain-La-Neuve, Belgium}

\email{gbelletti451@gmail.com}

\author{Renaud Detcherry}
\address{Institut de Mathématiques de Bourgogne \& Institut Universitaire de France, UMR 5584 CNRS, Université Bourgogne Franche-Comté, F-2100 Dijon, France}
\email{renaud.detcherry@u-bourgogne.fr}

\date{} 

\begin{document}
	
	\maketitle
	
	\begin{abstract} We prove that given two compact oriented $3$-manifolds $N$ and $M,$ with $M$ satisfying only a mild hypothesis, there is a hyperbolic $3$-manifold $N'$ arbitrarily ``closely related'' to $N,$ and such that $N'$ does not embed in $M.$ For instance, as a weak version of our main theorem, if $M$ is a rational homology sphere then for any $k\geq 1$ the $3$-manifold $N'$ can be chosen to be $Y_k$-equivalent to $N.$  Our techniques rely on the construction of $3$-manifolds with complicated Frohman--Kania-Bartoszy\'nska ideals, using strong approximation for $\mathrm{SO}_3$-Witten-Reshetikhin-Turaev quantum representations of mapping class groups of surfaces.
	\end{abstract}
	
\section{Introduction}
Broadly speaking, the goal of this article is to study the following question:
\begin{ques}\label{ques:mainProblem} ($3$-manifold embedding problem)
	If $N$ and $M$ are two compact oriented $3$-manifolds (possibly with boundary), how can one determine whether $N$ can be embedded inside $M?$
\end{ques}

Note that in the above, we make no assumption on the boundaries of $N$ and $M,$ which can have arbitrarily many connected components of arbitrary genera. 
We will restrict ourselves to connected manifolds; while embedding disconnected manifolds into each other presents some subtleties and cannot always be reduced straightforwardly to the connected case, to obstruct embeddings it suffices to do so for a single connected component. 

While Question \ref{ques:mainProblem} has rarely been stated in this general way in the literature, there are many techniques to approach this problem, coming from every corner of $3$-dimensional topology: among others, one can get obstructions from homology, as well as the hyperbolic (or simplicial) volume or prime and JSJ decompositions. On the other hand, algorithmic techniques arising for instance from normal surface theory can give partial answers to this question. We will give a comprehensive overview of classical techniques that have been applied to the $3$-manifold embedding problem in Section \ref{sec:overview}.

Because $3$-manifold topology as well as embeddings between manifolds of the same dimension are very rigid, it is expected that embeddings between $3$-manifolds are rare.

Our main results are a confirmation of this heuristic. We show that for any $N$ and $M$ (with some mild hypothesis on $M$) there exist $3$-manifolds $N'$
that are \textit{arbitrarily close} to $N$ but do not embed in $M.$ To make sense of the notion of being arbitrarily close, we use the so-called $Y_k$-equivalence relations on $3$-manifolds, as well as the profinite topology on profinite completions of fundamental groups.

For $F$ a compact oriented surface, we will denote by $I(F)$ its Torelli group, that is, the kernel of the natural map $\mathrm{Mod}(F)\longrightarrow \mathrm{Aut}(H_1(F;\Z)),$ where $\mathrm{Mod}(F)$ is the mapping class group of $F.$ For $G$ a group, we denote by $\Gamma_k G$ the $k$-th element of its lower central series, defined by $\Gamma_1 G=G$ and $\Gamma_{k+1}G=[\Gamma_kG,G].$

\begin{dfn}
	\label{def:Y_kequiv} For $k\geq 1,$ we say that two $3$-manifolds $N$ and $N'$ are $Y_k$-equivalent, and we write $N\underset{Y_k}{\sim}N'$ if one can obtain $N'$ from $N$ by cutting $N$ along an embedded compact oriented surface $F$ with at most one boundary component, and regluing using a gluing map in $\Gamma_k I(F).$
\end{dfn}
It is easy to see that two $Y_1$-equivalent manifolds have the same homology. The $Y_k$-equivalences (and the related $J_k$-equivalences) have been studied in depth in the literature; we refer to \cite{Mas:notes} for an overview of the subject. It is conjectured that if two $3$-manifolds are $Y_k$-equivalent for every $k,$ then they are homeomorphic (see \cite[Goussarov-Habiro Conjecture]{Mas:notes} and subsequent discussion); hence $Y_k$-equivalence gives us a good notion of $3$-manifolds being close.

For $G$ a group, we denote by $\widehat{G}$ its profinite completion, that is the profinite group defined as the inverse limit
$$\widehat{G}=\underset{N\triangleleft G, [G:N]<\infty}{\varprojlim} G/N.$$ Two finitely generated groups have isomorphic profinite completions if and only if they have the same set of finite quotients; in light of this, we will say that a sequence of groups $G_k$ \emph{profinitely converges} to $G$ if there is a sequence of natural numbers $d_k\ra \infty$ such that $G$ and $G_k$ have the same finite quotients up to order $d_k.$

It is generally believed that hyperbolic $3$-manifolds are \emph{profinitely rigid}, meaning that if $\widehat{\pi_1(M')}$ is isomorphic to $\widehat{\pi_1(M)}$ where $M$ is hyperbolic, then $M\simeq M'.$ It is known that closed hyperbolic $3$-manifolds \cite{Liu23}, hyperbolic $3$-manifolds with toral boundary \cite{Xu25} and closed Seifert fibered or graph manifolds \cite{Wil1,Wil2,Wil3} are \emph{profinitely almost rigid}, meaning that there are only a finite number of homeomorphism classes of $M'$ such that $\widehat{\pi_1(M')}\simeq \widehat{\pi_1(M)}.$

In this sense, having fundamental groups arbitrarily profinitely close to $\pi_1(N)$ gives another good notion of being close to $N.$

We are now able to state our main theorem:

\begin{thm}
	\label{thm:intro1} Let $N$ be a compact oriented $3$-manifold, and $M$ be a very good compact oriented $3$-manifold. Then, there exists a sequence $(N_k)_{k\geq 1}$ of hyperbolic compact oriented $3$-manifolds, with $N_k \underset{Y_k}{\sim} N$ and whose fundamental groups $\pi_1(N_k)$ profinitely converge to $\pi_1(N),$ and such that none of the $N_k$ embeds in $M.$ Furthermore, these manifolds can be chosen so that $\textrm{Vol}(N_k)$ diverges, and $N_k$ contains no essential, non-boundary parallel surface of genus less than $k$.
\end{thm}
The definition of \emph{very good $3$-manifold} is postponed to Definition \ref{def:VeryGood}; for the moment it suffices to say that any $3$-manifold that embeds in a rational homology sphere is very good, and that we conjecture that all compact oriented $3$-manifolds are very good (see also Proposition \ref{prop:VeryGood}).

Theorem \ref{thm:intro1} encapsulates both notions of a $3$-manifold $N$ being approximable by $3$-manifolds that are not embeddable in $M;$ we will package the two notions into a single, stronger equivalence relation, the $(Y_k,T_n)$-equivalence relation, which we will introduce and study in Section \ref{sec:definitons}. 

In fact, our techniques allow us to give a probabilistic version of Theorem \ref{thm:intro1}.
To state this, we use a variation of the Dunfield-Thurston \cite{DT06} model of random $3$-manifolds, which we describe below.

For $G$ a finite or countable group, $S$ a generating set of $G,$ and $\mu$ a probability measure on $S$ with support $S,$ we can take $x_1,\cdots,x_d$ independent identically distributed random variables with law $\mu$ and define the associated random walk $(g_d)_{d\geq 0}$ by $g_d=x_1\ldots x_d$.

 Now pick a Heegaard splitting $N=C\cup_\phi H$ where $H$ is a handlebody of genus $g\geq 2$, $C$ is a compression body, $\partial_- C\simeq \partial H \simeq F$ and $\phi\in \mathrm{Mod}(F).$ Then we let $N'$ be the random $3$-manifold $C\cup_{\phi\circ f} H,$ where $f\in \Gamma_kI(F)$ is given by a random walk in the group $\Gamma_k I(F)$ for some $k\geq 1.$ 

\begin{thm}
	\label{thm:main2} Let $k\geq 1$ and let $M$ be a very good $3$-manifold. Let also 
	$$N=C\cup_{\phi} H$$ be a Heegaard splitting of a compact oriented $3$-manifold of genus $g\geq 2$. Given the random $3$-manifold $$N'_d=C\cup_{\phi\circ f_d} H$$ where $f_d\in \mathrm{Mod}(F)$ is a random walk of length $d$ in $\Gamma_kI(F),$ we have:
	
	$$\underset{d\rightarrow \infty}{\liminf} \ \mathbb{P}(N'_d \ \textrm{does not embed in} \ M)>0.$$
\end{thm}

In fact, in Theorem \ref{thm:proba}, we give a version of this theorem that is both quantitative and more general. We note that our lower bound for the liminf of the probability does not depend on the choice of $S$ and $\mu,$ nor on $k.$

The main tool that will lead to Theorems \ref{thm:intro1} and \ref{thm:main2} are the Frohman--Kania-Bartoszy\'nska ideals. For $M$ a compact oriented $3$-manifold, the FKB ideals $I_p(M)$ are the ideals of localized cyclotomic rings $\Z[p^{-1},\zeta_p]$ generated by $RT_p(N)$ for all closed manifolds $N$ that contain $M$, where $RT_p(N)$ is the Witten-Reshetikhin-Turaev invariant of $N$ at a primitive $2p$-th root of unity (see Section \ref{sec:FKBideals} for more details). They were defined in \cite{FKB} to provide quantum obstructions to embedding between $3$-manifolds. 

In order to provide $3$-manifolds with suitable $I_p(M)$ ideals, we will make use of the \textit{strong approximation theorem} for quantum representations, which was proved by Masbaum and Reid \cite{MR12}. This theorem says that the quantum representations associated to the $\mathrm{SO}_3$-Witten-Reshetikhin-Turaev TQFT at roots of unity of prime order surject onto $\mathrm{PSL}_d(\mathbb{F}_q)$ for infinitely many $q$'s.

The paper is structured as follows. In Section \ref{sec:overview}, we go over what can be achieved with classical techniques such as homology and simplicial volume. In Section \ref{sec:definitons} we introduce the background tools needed for the statement and proofs of the main results. In Section \ref{sec:equivRelProp} we prove some properties of $Y_k$-equivalence and the related equivalence relations that will appear in the statements. In Section \ref{sec:strongApprox} we extend the strong approximation properties to subgroups of the mapping class group.In Section \ref{sec:proofs} we put all the ingredients together to prove Theorem \ref{thm:main}, which is a stronger version of Theorem \ref{thm:intro1}. Finally, in Section \ref{sec:proba}, we prove a quantitative and stronger version of Theorem \ref{thm:main2}.

\textbf{Acknowledgements:} We wish to thank Gregor Masbaum for helpful discussions. We also wish to thank Gabriele Viaggi for his help with the arguments of Proposition \ref{prop:hyperbolic}. The first named author was supported by FNRS (grant 1.B.044.25F). The IMB, host institution of the second named author, receives support from the EIPHI Graduate School (contract ANR-17-EURE-0002).


\section{Overview of classical techniques for the $3$-manifold embedding problem}
\label{sec:overview}

In this section, we describe a few obstructions to the embedding problem that can be stated in terms of classical or hyperbolic invariants, and describe briefly two algorithms that give partial answers to Question \ref{ques:mainProblem}.

The main purpose of this section is to illustrate what can be said about Question \ref{ques:mainProblem} from classical methods, as well as showcase some of their limitations.
Therefore, we will not try to achieve the most general statements, nor do we claim originality on any of these results. 

Not surprisingly, a first obstruction for $N$ to embed in $M$ may be obtained from homology. Indeed, from Mayer--Vietoris and the ``half lives, half dies'' principle, it is not difficult to prove:

\begin{lem}\label{lem:homolBound}\cite[Thm 1.3]{Ton11}
	If $N$ and $M$ are compact oriented $3$-manifolds and $N$ embeds in $M$ then for any field $\mathbb{F},$ one has
	$$b_1(N;\mathbb{F})-\frac{1}{2}b_1(\partial N;\mathbb{F})\geq b_1(M;\mathbb{F})-\frac{1}{2}b_1(\partial M;\mathbb{F}),$$
	and the inequality is sharp.
\end{lem} 
We refer to \cite[Thm 1.3]{Ton11} for a proof of this fact. It seems like sharper bounds using the torsion linking form may be obtained; however, the authors could not locate any such obstruction in the literature.

\medskip

When $\partial N$ has low genus, Question \ref{ques:mainProblem} becomes easier. For instance, if $\partial N=S^2,$ then Question \ref{ques:mainProblem} is equivalent to deciding whether $N$ is a connected summand of $M.$ In the case where $\partial N$ and $\partial M$ are unions of tori, there are also many tools at our disposal. We illustrate this with the following two obstructions, coming from hyperbolic geometry and from finite type invariants.

\begin{prop}
	\label{prop:GromovNormObs} Let $N$ be a compact oriented $3$-manifold with toroidal boundary and assume that $\partial N$ is incompressible. Let $M$ be an irreducible and atoroidal $3$-manifold with empty or toroidal boundary. Then if $N$ embeds in $M$ 
	$$||M||\leq ||N||$$
	where $||\cdot ||$ is the Gromov norm.
\end{prop}
\begin{proof}
	The hypotheses imply that if $N$ embeds in $M$ then $M$ is obtained from $N$ by Dehn filling some of its boundary components. However, by \cite{Gromov}, the Gromov norm decreases under Dehn filling.
\end{proof}
We note that Proposition \ref{prop:GromovNormObs} does not give any information when $M$ is a graph manifold.

The following proposition illustrates that finite type invariants may provide embedding obstructions in some specific situations.
\begin{prop}
	\label{prop:CassonObs} Let $N$ be a knot exterior in an integral homology sphere $M_0,$ with incompressible boundary. Then if $N$ embeds in an irreducible atoroidal integral homology sphere $M,$ one must have 
\[	\lambda(M)=\lambda(M_0) \pmod {\Delta_N''(1)}\]
	where $\lambda$ is the Casson invariant and $\Delta_N$ is the (symmetric version of the) Alexander polynomial of the knot whose exterior is $N.$
\end{prop}
\begin{proof}
	As in the proof of Proposition \ref{prop:GromovNormObs}, the hypotheses imply that if $N$ embeds in $M,$ then $M$ is actually a Dehn filling of $N.$ Moreover, $M$ being an integral homology sphere, we have $M=N(1/k)$ for some $k\in \Z.$ However, the surgery formula for the Casson invariant (see for instance \cite[Lecture 12]{Saveliev}) states that
	$$\lambda\left(N\left(\frac{1}{k+1}\right)\right)-\lambda\left(N\left(\frac{1}{k}\right)\right)=\Delta_N''(1),$$
	for any $k\in \Z.$ Therefore we get that 
	$$\lambda(M_0)=\lambda(N(1/0))=\lambda(M) \pmod {\Delta_N''(1)}.$$
\end{proof}

Geometric decompositions of $3$-manifolds may also be used to provide embedding obstructions. Recall that by Kneser's theorem, any closed $3$-manifold has a unique decomposition $M=M_1\#\ldots \# M_k$ where the $M_i$'s are prime $3$-manifolds. Moreover, the Jaco-Shalen-Johannson decomposition \cite{JS,J} states that any closed irreducible $3$-manifold has a unique decomposition along embedded tori into atoroidal or Seifert fibered $3$-manifolds.

\begin{prop}\label{prop:GeomDecObs} Let $M$ be a closed $3$-manifold and $N$ be a compact oriented $3$-manifold.
	\begin{enumerate}
		\item If $M$ is irreducible and $N=N_1\# N_2,$ then $N$ embeds in $M$ if and only if either $N_1$ embeds in $M$ and $N_2$ embeds in $S^3$ or vice versa.
		\item If $M$ is irreducible and atoroidal and $N=N_1\cup_{T^2} N_2$ is the JSJ decomposition of $N,$ then if $N$ embeds in $M$ then either $N_1$ or $N_2$ embeds in a solid torus and a Dehn filling of the other embeds in $M.$  
	\end{enumerate}
\end{prop}
\begin{proof}
	For the first claim, if $M=N\cup_{\partial N}N'$ where $\partial N'\simeq \partial N,$ then the essential sphere in $N$ bounds a ball $B^3$ on one side. If the ball lies on the $N_1$-side, then $N_1$ embeds in $B^3,$ and thus in $S^3.$ Moreover, we also get that $N_2$ may be filled to get $M,$ hence the claim. The converse is easy as $M\# S^3\simeq M.$
	
	For the second claim, the incompressible torus in $N$ becomes compressible in $M,$ hence bounds a solid torus on one side. Therefore, $N_1$ or $N_2$ embeds in a solid torus, and we also get that a Dehn filling of the other piece embeds in $M.$
\end{proof}

Obstructions using geometric decompositions become more difficult to state when the number of pieces increases. In particular, we note that there are $3$-manifolds with arbitrarily many JSJ pieces that embed in $S^3.$

To finish the overview of Question \ref{ques:mainProblem}, we quote two recent results that provide partial algorithmic answers to the embedding problem:

\begin{thm}
	\label{thm:algo}
	\begin{enumerate}
		\item \cite[Theorem 1.3]{MST18} Given a compact oriented $3$-manifold $N,$ there exists an algorithm that decides whether $N$ embeds in $S^3.$
		\item \cite{Sch25} Given two compact oriented $3$-manifolds $N,M$ with toroidal or empty boundary, there exists an algorithm that decides whether $M$ is a Dehn filling of $N.$
	\end{enumerate}
\end{thm}
The algorithm of \cite{MST18} uses the theory of $0$-efficient triangulations of $3$-manifolds developed in \cite{JR03}. The main technical step is to show that if $N$ embeds in $S^3,$ there exist a curve $\gamma \subset \partial N$ of bounded length which bounds a disk in $S^3.$

A rough description of the algorithm of \cite{Sch25} is as follows. Dehn filling along non-short slopes or non-fibered slopes respects the JSJ decomposition, and Dehn fillings along long slopes of hyperbolic pieces create hyperbolic manifolds with small injectivity radius. In the case of one JSJ piece for $N,$ one can then reduce the study to Dehn fillings along a finite set of slopes and apply Kuperberg's algorithm \cite{Kup19} for the homeomorphism problem for $3$-manifolds. The case where $N$ has more JSJ pieces requires in addition a delicate analysis of a decidable system of Diophantine equations.

Both algorithms above are computationally difficult to apply in practice, and hard to use to establish results on infinite families of $3$-manifolds.

Meanwhile, one can see that the obstructions described in this section are restricted, requiring hypotheses on the genus of $\partial N$ or the geometric decomposition of $M.$ We will see that the Frohman--Kania-Bartoszy\'nska obstructions that we will introduce in the next section retain the flexibility of the homological obstruction of Lemma \ref{lem:homolBound}, while being more powerful.

\section{Definitions and background}
\label{sec:definitons}
\subsection{Equivalence relations on $3$-manifolds}
\label{sec:equivRel}
\begin{dfn}
	Any element $\phi\in \mathrm{Mod}(\Sigma)$ induces an action on $H_1(\Sigma,\Z)$. The \emph{Torelli subgroup} of $\mathrm{Mod}(\Sigma)$, denoted by $I(\Sigma)$, is the subgroup formed by elements acting trivially (i.e., like the identity) on $H_1(\Sigma,\Z)$. In other words, it is the kernel of the natural map $\mathrm{Mod}(\Sigma)\ra \textrm{Aut}(H_1(\Sigma;\Z))$.
\end{dfn}

A filtration on the Torelli group is given by its \emph{lower central series}. Recall that the lower central series of a group $G$, denoted by $\Gamma_{k}G$, is defined recursively by $\Gamma_1 G:= G$, $\Gamma_{k+1}G=[\Gamma_k G,G]$. 

We will be interested in another family of subgroups of $\mathrm{Mod}(\Sigma):$ the subgroup $T_n\Sigma$ generated by all $n$-th powers of Dehn twists. Since conjugates of Dehn twists are themselves Dehn twists, $T_n\Sigma$ is also a normal subgroup of $\mathrm{Mod}(\Sigma).$ So as to not overcomplicate notation, we will usually drop the $\Sigma$ when talking about this subgroup and simply refer to it as $T_n$. When $\Sigma$ is closed of genus $g\geq 2,$ a theorem of Funar \cite{Fun99} shows that the subgroup $T_n$ is of infinite index in $\mathrm{Mod}(\Sigma)$ for any odd $n\geq 5;$ see also \cite{Mas98}.

A way of producing a new $3$-manifold $M'$ out of a $3$-manifold $M$ is by a local modification along a properly embedded surface $\Sigma$ called a \emph{Torelli twist}:

\begin{dfn}\label{def:TorelliTwist}
	Let $M$ be a compact oriented $3$-manifold, $\Sigma\subseteq M$ be a properly embedded compact oriented surface with at most one boundary component, and $\phi\in I(\Sigma)$. The \emph{Torelli twist} of $M$ along $\Sigma$ by $\phi$ is the manifold 
	\[M_\phi:= \left(M\setminus U(\Sigma)\right)\cup_{\widetilde{\phi}} U(\Sigma)\]
	where $U(\Sigma)$ is a regular neighborhood of $\Sigma$ and $\widetilde{\phi}$ is the diffeomorphism of $\partial U(\Sigma)$ given by $\phi$ on $\Sigma\times \{1\}$ and the identity on the other component.
\end{dfn}

The notion of Torelli twist allows one to define equivalence relations on $3$-manifolds, depending on the subgroups from which the maps $\phi$ are taken. A standard equivalence relation of this kind is the $Y_k$-equivalence:

\begin{dfn}\label{def:YkEquiv}
We say that two compact oriented $3$-manifolds $M,M'$ are $Y_k$-equivalent, and write $M\underset{Y_k}{\sim}M',$ if $M'$ can be obtained from $M$ by a sequence of Torelli twists along surfaces $\Sigma_1,\ldots \Sigma_d,$ and such that at each step, the gluing map is $\phi_i \in \Gamma_k I(\Sigma_i).$
\end{dfn}
With this definition, it is clear that $\underset{Y_k}{\sim}$ is an equivalence relation. However, a folklore result is that a single Torelli twist along a Heegaard surface is sufficient to relate two $Y_k$-equivalent $3$-manifolds.

In this paper, we will work with some finer equivalence relations than $Y_k$-equivalence, and which (to the authors' knowledge) have not been considered before.

\begin{dfn}\label{def:YkTnEquiv}
	Let $k\geq 1$ and let $n\geq 2$ be integers.
	We say that two compact oriented $3$-manifolds $M,M'$ are $(Y_k,T_n)$-equivalent, and write $M\underset{(Y_k,T_n)}{\sim}M',$ if $M'$ can be obtained from $M$ by a sequence of Torelli twists along surfaces $\Sigma_1,\ldots \Sigma_d,$ and such that at each step, the gluing map is $\phi_i \in \Gamma_k I(\Sigma_i)\cap T_n(\Sigma_i).$
\end{dfn}

We will also sometimes refer to $T_n$-equivalence as the relation obtained when we just require that the gluing maps $\phi_i\in T_n(\Sigma_i).$

We will look at one final variation of these equivalence relations on $3$-manifolds. 

\begin{dfn}\label{def:Hequiv}
Suppose $\mathcal{H}:=\left(H_{g}\right)_{g\geq 1}$ is a family of finite index, normal subgroups of $\mathrm{Mod}(\Sigma_{g})$.

Then we say that $M$ is $\mathcal{H}$-equivalent to $M'$ if $M'$ can be obtained from $M$ by a sequence of Torelli twists along surfaces $\Sigma_{g_1},\ldots, \Sigma_{g_d},$ with the $i$-th gluing map $\phi_i$ belonging to $H_{g_i}.$
\end{dfn}

Our main choice of finite index subgroups will come from kernels of Dijkgraaf-Witten TQFT representations. Fix a  finite collection of finite groups $\mathcal{G}=\lbrace G_1,\ldots ,G_r \rbrace$; we can set 
\[H_g=\underset{i=1}{\overset{n}{\bigcap}} \mathrm{Ker} \rho_{g,G_i},\]
where $\rho_{g,G_i}$ is the representation of $\mathrm{Mod}(\Sigma_g)$ coming from the untwisted Dijkgraaf-Witten TQFT with gauge group $G$ (see Section \ref{sec:DW-TQFT} for an overview of those TQFTs).

\begin{dfn}\label{def:YkTnGEquiv}
	Let $k\geq 1, n\geq 2$ be integers and let $\mathcal{G}=\lbrace G_1,\ldots ,G_r\rbrace$ be a finite collection of finite groups.
	We say that two compact oriented $3$-manifolds $M,M'$ are $(Y_k,T_n,\mathcal{G})$-equivalent, and write $M\underset{(Y_k,T_n,\mathcal{G})}{\sim}M',$ if $M'$ can be obtained from $M$ by a sequence of Torelli twists along surfaces $\Sigma_1,\ldots \Sigma_d,$ and such that at each step, the gluing map is 
	\[\phi_i \in \Gamma_k I(\Sigma_i)\cap T_n(\Sigma_i)\cap \left(\underset{j=1}{\overset{m}{\bigcap}}\mathrm{Ker} \rho_{G_j,\Sigma_i}\right).\]
\end{dfn}

\subsection{Generalities on TQFTs}
Let $\mathcal{C}_n$ be the category of $n+1$-cobordisms; in other words, the objects of $\mathcal{C}_n$ are $n$-dimensional,smooth oriented closed manifolds, and the morphisms between two objects are given by cobordisms between the manifolds (up to diffeomorphisms that are the identity on the boundary). This category has a natural monoidal structure (i.e. a tensor product) given by disjoint union, with the unit object given by the empty manifold, and a natural duality given by orientation reversal.

\begin{dfn}[Atiyah-Segal axioms for a TQFT]	
	Let $R$ be a commutative ring with $1$ and with a conjugation. A \emph{($n+1)$-topological quantum field theory} is a rigid monoidal functor $Z$ from $\mathcal{C}_n$ to the category of $R$-modules with the additional property that for each cobordism $M$, $Z(M^*)=\overline{Z(M)}$.	
\end{dfn}

An $(n+1)$-dimensional TQFT provides two major tools: an invariant of closed $n+1$-manifolds and a representation of mapping class groups of $n$-dimensional manifolds. For the former, a closed $n+1$-dimensional manifold $M$ is a cobordism between the empty manifold and itself; therefore $Z(M)$ is a map from $R$ to itself, i.e. multiplication by an element of $R$. For the latter, a diffeomorphism $\phi: \Sigma\ra \Sigma$ gives a mapping cylinder $T_\phi$ which in turn gives a map $Z(\phi): Z(\Sigma)\ra Z(\Sigma)$. Because of the invariance under diffeomorphism of $Z$, this induces a map $\rho_Z$ from the mapping class group of $\Sigma$ to the endomorphisms of $Z(\Sigma)$. 

Given two manifolds $M_1,M_2$ with boundary $\Sigma$, we can consider $Z(M_1\cup \overline{M_2})$, with the gluing happening along the natural identification of the boundaries. This extends to a bilinear pairing $\langle\cdot,\cdot\rangle$ on $Z(\Sigma)$; in the TQFTs we will consider, this pairing is non-degenerate and it is straightforward to check that it is preserved by $\rho_Z$. Furthermore, if $M= M_1\cup_\phi M_2$, with $\phi$ an identification between $\partial M_1$ and $\partial M_2$, the axioms of TQFT immediately give $Z(M)=\langle Z(M_1),\rho_Z(\phi)Z(M_2)\rangle$. In a similar way, if $M$ is a mapping torus of a diffeomorphism $\phi:\Sigma\ra \Sigma$, then $Z(M)=\mathrm{tr} \rho_Z(\phi)$.

We will only be interested in $2+1$ TQFTs, and specifically we will look at two kinds: the Dijkgraaf-Witten TQFTs (see Subsection \ref{sec:DW-TQFT}) and the $SO(3)$ Witten-Reshetikhin-Turaev TQFTs \cite{BHMV}. In the former case, the base ring will be $\C$; in the latter, it will be $\Z\left[p^{-1},\zeta_p\right]$.

\begin{rem}
	Technically, the $SO(3)$ Witten-Reshetikhin-Turaev TQFT is not a TQFT in the Atiyah-Segal sense; rather, manifolds have an additional weak datum called a $p_1$-structure. In practice, this can be ignored except for the fact that the mapping class group representations are actually projective representations, i.e. they are maps from $Mod(\Sigma)$ to the projective linear group. 
\end{rem}

\subsection{TQFT representations of the mapping class group}
\label{sec:quantumRep}

We denote by $RT_p$ the functor given by the $SO(3)$ Witten-Reshetikhin-Turaev TQFT at prime level $p$.

The mapping class group of a connected $\Sigma_g$ acts projectively on $RT_p(\Sigma_g)$; in other words, we have a map $$\rho_p: Mod(\Sigma_g)\ra \mathrm{PSL}_{d_{p,g}}(\Z\left[p^{-1},\zeta_p\right])$$
where $d_{p,g}$ is the dimension of $RT_p(\Sigma_g)$; it is easy to see that $d_{p,g}$ is increasing in $g$. Because of the TQFT axioms, this map preserves a Hermitian form and so it actually has image in $PSU_{d_{p,g}}$.

In \cite{MR12}, Masbaum and Reid established that for $g\geq 2$ and $p\geq 5$ a prime, the reductions of the quantum representation $\rho_p$ of $\mathrm{Mod}(\Sigma_g)$ modulo infinitely many maximal ideals of $\Z[\zeta_p]$ are surjective; the latter result is sometimes referred to as \emph{strong approximation} for the $\mathrm{SO}_3$ WRT quantum representations.

\begin{thm}\cite[Theorem 1.2]{MR12}\label{thm:MasbaumReid}
	For every $g\geq 2$, there exist infinitely many primes $q$ such that $\rho_{p,q}: Mod(\Sigma_g)\ra \mathrm{PSL}_{d_{p,g}}(\mathbb{F}_q)$ is surjective, where $\rho_{p,q}$ is equal to $\rho_p \pmod{ J}$ with $J$ a maximal ideal of $\Z\left[p^{-1},\zeta_p\right]$ such that $\Z\left[p^{-1},\zeta_p\right]/J=\mathbb{F}_q$.
\end{thm}
\begin{rem}
	\label{rk:rho_pq} The representations $\rho_{p,q}$ above actually depend on a choice of such a maximal ideal $J;$ but we will not mention $J$ in our notations.
\end{rem}
\begin{rem}\label{rk:strongApprox}
	The statement we gave of strong approximation is somewhat different from the one given in \cite{MR12}. Firstly, we do not assume that $p$ is $3$ modulo $4$. Secondly, our base ring is $\Z\left[p^{-1},\zeta_p\right]$ since we are not using the integral version of the TQFT (see also Remark \ref{rem:normalization}). The fact that this statement is also valid is addressed in \cite[Remarks 3.1 and 3.5]{MR12}. 
\end{rem}

\subsection{Quantum invariants and the Frohman--Kania-Bartoszy\'nska ideal}
\label{sec:FKBideals}
Let $p$ be an odd prime, and let $\zeta_p$ be a primitive $2p$-th root of unity. For a closed, oriented (not necessarily connected) $3$-manifold $M$, we denote by $RT_p(M)$ the $SO(3)$ Witten-Reshetikhin-Turaev invariant of $M$ at level $p$ and evaluated at $A=\zeta_p$. Here we use the TQFT normalization for $RT_p(M)$; in this case $RT_p(S^3)=\frac{\zeta_p^2-\zeta_p^{-2}}{\sqrt{-p}}$ which is a unit in $\Z\left[p^{-1},\zeta_p\right]$ (see for example \cite[Lemma 4.1(ii)]{GMvW}). By \cite{MR97}, $RT_p(M)\in \Z[p^{-1},\zeta_p]$. If $M$ is compact with boundary, we will write $RT_p(M)$ to denote the vector given by $M$ in $RT_p(\partial M)$, the vector space associated to $\partial M$ in the Witten-Reshetikhin-Turaev TQFT.

\begin{dfn}\label{def:VeryGood}
	We say that a compact oriented $3$-manifold $M$ is \emph{$p$-good} for some prime $p\geq 5$ if $RT_p(M)\neq 0$. We say that it is \emph{very good} if it is $p$-good for all but finitely many primes.
\end{dfn}
\begin{prop}\label{prop:VeryGood}
	A compact manifold $M$ is very good if any of the following conditions holds:
	\begin{enumerate}
		\item $M$ is a rational homology sphere;
		\item $M=M_1\# M_2$ and both $M_1,M_2$ are very good;
		\item $M$ is the double of a very good manifold;
		\item $M$ is a submanifold of a very good manifold; 
		\item $M$ is obtained from a very good manifold by attaching a connected Seifert cobordism with one input and at least one output boundary components.
	\end{enumerate}
	
	Furthermore, if $M$ is hyperbolic, for any prime $p,$ there exists a finite cover of $M$ which is $p$-good.
\end{prop}
\begin{proof}
	The fact that rational homology spheres are very good follows from Murakami's theorem \cite{Mur95} that $$RT_p(M)=|H_1(M;\Z)| \pmod {\zeta_p^2-1}$$ when $M$ is a rational homology sphere. The fact that for any $3$-manifolds $M_1,M_2$ one has
	$$RT_p(M_1\# M_2)=\frac{RT_p(M_1)RT_p(M_2)}{RT_p(S^3)}$$ shows that being very good is closed under connected sum; similarly the fact that $RT_p(DM)=|| RT_p(M)||^2$ (where $DM$ is the double of $M$ and $||\cdot ||$ is the norm for the natural Hermitian form on $RT_p(\partial M)$) shows closure under doubles.
	
	Given a cobordism $C$ between two (not necessarily connected) surfaces $\partial_+C$ and $\partial_-C$, we have an induced TQFT map $RT_p(C):RT_p(\partial_-C)\ra RT_p(\partial_+C)$. Let $M$ be a submanifold of a very good manifold $M'$; then there is a cobordism $C$ such that $M'= M\cup C$, which implies $RT_p(M')=RT_p(C)\cdot RT_p(M)$, and if $M'$ is very good then $M$ must also be.
	
	Furthermore, let $C$ be a Seifert manifold with $n\geq 2$ boundary components. Viewing $C$ as a cobordism 
	$$C:\partial_-C \longrightarrow \partial_+ C$$
	where $\partial_-C \simeq T^2$ and $\partial_+C\simeq \underset{1\leq i \leq n-1}{\coprod} T^2$, it was shown in \cite{DKM25} that $RT_p(C)$ is injective as a map $RT_p(\partial_-C)\ra RT_p(\partial_+ C)$ for all but finitely many primes $p$ (specifically, all primes $p$ that do not divide the orders of the exceptional fibers), proving Claim (5).
	
	Finally, if $M$ is hyperbolic, it has a finite cover $T_{\phi}$ that is the mapping torus of some $\phi\in Mod(\Sigma_g)$, by \cite{Agol}. The images of $\phi^k$ (for $k\in \Z$) under the TQFT representation $\rho_p$ form a subgroup of $PSU_{d_{p,g}}$ where $d_{p,g}=\dim RT_p(\Sigma_g)$. If this subgroup is finite, it means there is $k$ such that $RT_p(M_{\phi^k})=\mathrm{tr}\rho_p(\phi^k)=d$, and $M_{\phi^k}$ is $p$-good.
	
	 If instead this subgroup is infinite, by compactness of $\mathrm{PSU}_{d_{p,g}}$, it must have $Id_{RT_p(\Sigma)}$ as an accumulation point, which means there is $k$ big enough so that $\rho_p(\phi^k)$ is arbitrarily close to the identity; by continuity of the trace, $\lvert RT_p(M_{\phi_k})\rvert =\lvert\mathrm{tr}\rho_p(\phi^k)\rvert>d-\epsilon$, and $M_{\phi^k}$ is $p$-good.
\end{proof}

\begin{rem}\label{rk:allVeryGood}
The authors conjecture that every $3$-manifold is very good. In fact, we note that the Chen-Yang volume conjecture \cite{CY18} implies that every $3$-manifold is very good.
\end{rem} 

\begin{dfn}\label{def:FKBideals}\cite{FKB}
	Let $N$ be a compact oriented $3$-manifold. The \emph{$p$-th Frohman--Kania-Bartoszy\'nska ideal of $N$}, denoted with $I_p(N)$, is the ideal of $\Z[p^{-1},\zeta_p]$ generated by $RT_p(M)$ for all $M$ closed connected oriented $3$-manifold containing $N$.
\end{dfn}

It is clear from the definition that if $N\subseteq N'$, necessarily $I_p(N')\subseteq I_p(N)$, since every closed manifold containing $N'$ must also contain $N$. This, coupled with the fact that for a closed manifold $M$, $I_p(M)$ is generated by $RT_p(M)$, can be used to provide obstructions to embeddability. In particular, since $I_p(S^3)=\Z[p^{-1},\zeta_p]$, if $N$ is such that $I_p(N)$ is a proper ideal, then $N$ cannot embed in $S^3$.

\begin{rem}\label{rem:normalization}
	The original definition of $I_p$ in \cite{FKB} was given using the integral normalization of $RT_p$; in this setting, $RT_p(S^3)=1$, $RT_p(M)\in \Z\left[\zeta_p\right]$ and thus $I_p(M)$ would be an ideal in $\Z[\zeta_p]$. We use this (a priori) slightly weaker version because it makes it less cumbersome to apply the TQFT constructions we will need to prove our results.
\end{rem}

\subsection{Dijkgraaf-Witten TQFT}
\label{sec:DW-TQFT}
Through the Turaev-Viro construction, any spherical fusion category gives rise to a TQFT; for an introduction to the topic, see \cite{Turaev}. When applied to the category of representations of a finite group $G$, this construction gives rise to the \emph{untwisted Dijkgraaf-Witten TQFT with gauge group $G$} \cite{DW90} \cite{FQ93}. These can be considered the simplest examples of TQFTs, and the following nice properties (plus the basic properties of TQFTs) are all we will need to know about it.

\begin{itemize}
	\item For a closed oriented $3$-manifold $M$, $Z_G(M)=\frac{1}{\lvert G\rvert}\lvert \mathrm{Hom}(\pi_1(M),G)\rvert$;
	\item For a closed oriented surface $\Sigma$, $Z_G(\Sigma)=\C[\mathcal{X}_G(\Sigma)]$, where $\mathcal{X}_G$ is the set of conjugacy classes of maps from $\pi_1(\Sigma)$ to $G$;
	\item the action of $\mathrm{Mod}(\Sigma)$ on $Z_G(\Sigma)$ is induced by its action on $\pi_1(\Sigma)$.
\end{itemize}

\begin{lem}\cite[Lemma 7.10]{Fun13}\label{lem:Funar}
	If $M_1$ and $M_2$ are oriented closed $3$-manifolds such that $Z_G(M_1)=Z_G(M_2)$ for all finite groups $G$, then $\pi_1(M_1)$ and $\pi_1(M_2)$ have isomorphic profinite completions.
\end{lem}

The Dijkgraaf-Witten TQFT provides representations of $\mathrm{Mod}(\Sigma)$ into any finite group $G$; we denote these by $\rho_G$. Clearly $\ker \rho_G$ is a finite index normal subgroup of $\mathrm{Mod}(\Sigma)$.

\begin{lem}\label{lemma:Gequiv-DWinvariants}
	Let $M_1$ and $M_2$ be two closed compact oriented $3$-manifolds. If $M_1\underset{G}{\sim}M_2$, then $Z_G(M_1)=Z_G(M_2)$.
\end{lem}
\begin{proof}
	Suppose $M_2$ is obtained from $M_1$ by a Torelli twist along the closed surface $\Sigma$ by the map $\phi\in \ker \rho_G$. Suppose first that $\Sigma$ is separating; call $H_1$ and $H_2$ the two components of $M\setminus U(\Sigma)$, where $U(\Sigma)$ is an open tubular neighborhood of $\Sigma$. Then by the TQFT properties:
	$$Z_G(M_1)=\langle Z_G(H_1),Z_G(H_2)\rangle=\langle Z_G(H_1),\rho_G(\phi) Z_G(H_2)\rangle=Z_G(M_2).$$
	If $\Sigma$ is instead non-separating, then viewing $C=M_1\setminus \Sigma$ as a cobordism $\Sigma\longrightarrow \Sigma,$ we have
	$$Z_G(M_1)=\mathrm{Tr}(Z_G(C))=\mathrm{Tr}(\rho_G(\phi)Z_G(C))=Z_G(M_2).$$
\end{proof}

The following proposition relates the $T_n$-equivalence to $\mathrm{Ker} \rho_G$ equivalence for a finite group $G.$ Recall that the exponent $e(G)$ of a finite group $G$ is the least integer $d$ such that $g^d=e_G$ for all $g\in G.$

\begin{lem}
	Let $G$ be a finite group, let $n$ be a multiple of $e(G)$ and let $M$ and $M'$ be two closed compact oriented $3$-manifolds. Then 
	$$M\underset{T_n}{\sim}M'\Longrightarrow M\underset{G}{\sim}M'.$$
\end{lem}

\begin{proof}
	The TQFT axioms for the Dijkgraaf-Witten TQFT of gauge group $G$ imply that it suffices to show that $T_n(\Sigma)\subset \mathrm{Ker}\rho_{G,\Sigma}$ for any surface $\Sigma.$ Let $\gamma\in \pi_1(\Sigma),$ and let $t_c$ be the Dehn twist along a simple closed curve $c\subset \Sigma.$ We choose the base point for $\Sigma$ on $c,$ and isotope $\gamma$ so that it intersects $c$ transversely and only in this base point (but perhaps several times). In other words, we can write $\gamma=\gamma_1\gamma_2\ldots \gamma_d,$ and we have $t_c^n(\gamma)=\gamma_1 c^{\pm n} \gamma_2 c^{\pm n} \ldots \gamma_d.$ It is clear that any representation $\pi_1(\Sigma)\longrightarrow G$ sends $\gamma$ and $t_c^n(\gamma)$ to the same element. Hence $t_c^n \in \mathrm{Ker} \rho_G$ and since the latter is normal in $\mathrm{Mod}(\Sigma),$ we have $T_n(\Sigma)\subset \mathrm{Ker}\rho_{G,\Sigma}.$
\end{proof}

Finally we write down a slight variation of Lemma \ref{lem:Funar}, whose proof is essentially word-for-word the same as its original proof.

\begin{lem}
	Suppose $M_1$ and $M_2$ are two manifolds such that $Z_G(M_1)=Z_G(M_2)$ for all finite groups $G$ of order less than or equal to $k$. Then $\pi_1(M_1)$ and $\pi_1(M_2)$ have the same quotients of order at most $k$.
\end{lem}

Combining everything we have seen so far, we immediately obtain the following.

\begin{cor}\label{cor:convergence}
	Take $\mathcal{G}_k$ the set of all groups (up to isomorphism) of order at most $k$. Suppose  we have a sequence $M_k\sim_{\mathcal{G}_k} M$, for $k\in \N$; then $\pi_1(M_k)$ profinitely converges to $\pi_1(M)$.
\end{cor}

\section{Properties of the $(Y_k,T_n)$-equivalence}
\label{sec:equivRelProp}

This short section establishes two further properties of the $(Y_k,T_n)$-equivalence. The first one, which will be used in the proof of Theorem \ref{thm:main} is the following:
\begin{prop}
	\label{prop:hyperbolic} Let $M$ be a compact oriented $3$-manifold, let $v>0$ and $g\geq 1$. Then for any $k,n\geq 1,$  there exists a hyperbolic $3$-manifold $M'$ such that $M\underset{(Y_k,T_n)}{\sim} M',$ and such that the volume of $M'$ is greater than $v$ and $M'$ contains no essential, non-boundary parallel surfaces of genus less than $g$.
\end{prop}
The proof of Proposition \ref{prop:hyperbolic} uses techniques that are standard but quite different from the rest of the paper; the volume condition in particular requires a very technical proof, and was explained to us by Gabriele Viaggi.
\begin{proof}
	Consider a Heegaard splitting $M=C \underset{\Sigma}{\cup} H$; denote by $V$ the set of curves in $\Sigma$ bounding disks in $H$.
	
	Note that the subgroup
	$$N=\Gamma_k I(\Sigma)\cap T_n(\Sigma)$$
	of $\mathrm{Mod}(\Sigma)$ is a non-trivial non-central normal subgroup of $\mathrm{Mod}(\Sigma_g)$, as follows for instance from Proposition \ref{prop:densityGammakI}. 
	
	 By \cite{Hempel}, the Hempel distance of the Heegaard splitting 
	$$M_m=C\underset{\phi^m}{\cup} H$$
	is unbounded (indeed, it grows linearly by \cite{AS}), provided that $\phi$ satisfies a genericity assumption: the stable and unstable laminations of $\phi$ must avoid the closure of $V$ in the set of projective measured laminations. By \cite[Lemma 2.6]{Long86}, \cite[Theorem 1.2]{Mas86} and \cite[Lemma 3.4]{Mah}, we can always find $\phi\in N$ that satisfies this genericity assumption.
	
	 Therefore, for each $D$, we can find $m$ such that the Hempel distance of the splitting given by $\phi^m$ is bigger than $D$. We show that for any $v,g$ there is an $m$ such that $M_m$ satisfies the requirements of this proposition.
	
	First of all, by \cite{Sch}, if the Hempel distance of a splitting of $M_m$ is $D$, then $M_m$ cannot contain any incompressible, non-boundary parallel surface of genus less than $D/2$. This implies that if $D$ is big enough, there are no essential tori, annuli, spheres or disks, which by geometrization implies that $M_m$ is hyperbolic.
	
	Secondly, by \cite[Theorem 8.1]{BMNS}, the volume of $M_m$ can be arbitrarily high. To see this, take two copies of $H$ and $\widetilde{C}$ the double of $C$ along the components of its lower boundary of genus higher than $1$; $\widetilde{C}$ is irreducible and atoroidal, since $C$ itself is atoroidal, irreducible and does not contain any essential annuli or disks whose boundary lies entirely in its lower boundary. We refer the interested reader to \cite{BMNS} for details and notation. We fix a marking $V$ on $\partial H$ given by a disk system, and a marking $V'$ on the upper boundary of $C$ given by the marking of $H$ through the natural identification coming from the Heegaard splitting of $M$. We need to check that there is an $R$ such that the gluing given by $\phi^m$ has $R$-bounded combinatorics for all $m$. We first fix a marking $\mu$ on $\partial_+C$ corresponding to a hyperbolic surface belonging to the axis of $\phi$; the fact that there is an $R$ such that $\mu$ and $\phi^m\mu$ have $R$-bounded combinatorics for all $m$ follows from the arguments of \cite[Section 2.14]{BMNS}. To check that the gluing itself has $R$-bounded combinatorics, we need to check that for any subsurface $Y$, the distance $d_Y(V,\phi^m V')$ is bounded by some $R_0$ independent of $Y$ and $m$. We have that $$d_Y(V,\phi^m V')\leq d_Y(V,\mu)+d_Y(\mu,\phi^m\mu)+d_Y(\phi^m\mu,\phi^m V')$$ and \begin{itemize}
		\item $d_Y(V,\mu)$ is bounded by some constant due to \cite[Theorem 6.12]{MM2};
		\item the middle term is bounded by $R$ because $\mu,\phi^m \mu$ has $R$-bounded combinatorics;
		\item $d_Y(\phi^m\mu,\phi^m V')$ is equal to $d_Y(\mu,V')$ because $\phi^m$ acts as an isometry and thus it is again bounded by some constant due to \cite[Theorem 6.12]{MM2}.
	\end{itemize}
	
	Furthermore the Hempel distance of the gluing $\phi^m$ grows linearly (this is referred to as the height in \cite{BMNS}). Then for $m$ big enough $\phi^m$ satisfies the hypotheses of \cite[Theorem 8.1]{BMNS}, which means that we can glue $\widetilde{C}$ with the two copies of $H$ using $\phi^m$ to obtain a hyperbolic manifold without boundary homeomorphic to $M_m$ doubled along the boundary components of genus higher than $1$, denoted with $DM_m$. By Mostow rigidity, the volume of $DM_m$ must be twice the volume of $M_m$. Moreover, there is a $K$ such that $DM_m$ is $K$-bilipschitz to a given ``model manifold'' built from 5 pieces equipped with a geometric structure: the two copies of $H$ glued to $\widetilde{C}$ via two $I$-bundles. The geometric structure on the $I$-bundles depends on $m$, and their volumes grow at least linearly in $D$, while those of the remaining pieces do not. This implies that the volume of $DM_m$ grows at least linearly in $m$; the same must be true for $M_m$ as well.

	We can thus conclude that for all sufficiently large $m$, the manifolds $M_m$ are $(Y_k,T_n)$-equivalent to $M$, are hyperbolic, have volume larger than a given $v$, and contain no essential, non-boundary parallel surfaces of genus less than $g$.

\end{proof}

\begin{rem}
	Using the same techniques, we can show that for any $M$, there is a hyperbolic $M'$ that is $(Y_k,T_n,\mathcal{H})$-equivalent to $M$ for any collection of finite index, normal subgroups as in Definition \ref{def:Hequiv}. Indeed, because the groups in $\mathcal{H}$ are finite index, for any $\phi$ there must be an $m_0$ such that $\phi^{m_0}\in \mathcal{H}$; then following the notations of the proof of Proposition \ref{prop:hyperbolic}, the $3$-manifold $M_{km_0}$  for $k$ big enough will work.
\end{rem}

Our second property will not be needed for the proof of our main theorems, but gives a justification why $3$-manifolds that are $T_{n!}$-equivalent for large $n$ may be considered close.
\begin{lem}\label{lem:Tnequiv}
	Suppose $M$ and $M'$ are two compact oriented $3$-manifolds that are $T_n$ equivalent for some $n$, and suppose that $Z$ is a TQFT such that $Z(\tau)^n=I$ for every Dehn twist $\tau$. Then $Z(M)=Z(M')$. In particular, if $M$ and $M'$ are $T_p$-equivalent for a prime $p$, $I_p(M)=I_p(M')$.
\end{lem}
\begin{proof}
	This lemma can be seen as a generalization of Lemma \ref{lemma:Gequiv-DWinvariants}, and the proof is identical, replacing the Djikgraaf-Witten TQFT by the TQFT $Z$. The last part follows from applying this to $Z=RT_p$, since $RT_p(M)=RT_p(M')$ implies that $I_p(M)=I_p(M').$
\end{proof}

If two manifolds $M$ and $M'$ are $T_n$-equivalent for every $n$, then in particular they have the same Witten-Reshetikhin-Turaev invariants at every level and the same Dijkgraaf-Witten invariants for every finite group. Similarly, if they are $\ker \rho^G$-equivalent for every finite group $G$, they have the same Dijkgraaf-Witten invariants; either way, Lemma \ref{lem:Funar} implies that $\pi_1(M)$ and $\pi_1(M')$ have isomorphic profinite completions. We remark that because of \cite{Liu23}, if $M$ is hyperbolic there are only a finite number of manifolds whose fundamental group has the same profinite completion as $\pi_1(M)$.

\section{Strong approximation for the restrictions of the quantum representations $\rho_p$ to subgroups}
\label{sec:strongApprox}

In this section, we will establish that the strong approximation of Masbaum-Reid from Section \ref{sec:quantumRep} still holds for the restriction of $\rho_p$ to many subgroups of $\mathrm{Mod}(\Sigma_g).$

For $g\geq 1$ and $p\geq 5$ a prime number, let $d_{p,g}=\dim RT_p(\Sigma_g),$ where $\Sigma_g$ is the closed surface of genus $g.$ Recall that $\rho_p$ preserves a Hermitian form and thus its image is in $PSU_{d_{p,g}}$. We start with the following lemma:

\begin{lem}\label{lem:SimplePerfect}
	Let $\Sigma$ be a closed surface of genus at least $2$ and let $N\triangleleft \mathrm{Mod}(\Sigma).$  Then:
	
	\begin{enumerate}[(a)]

		\item Either $\rho_p( N)$ is trivial or $\rho_p(\Gamma_k N)$ is dense in $PSU_{d_{p,g}}$ for any $k\geq 1.$
		\item Let $q=1 \pmod{ p}$ be a prime such that $\rho_{p,q}(\mathrm{Mod}(\Sigma))=\mathrm{PSL}_{d_{p,g}}(\mathbb{F}_q).$ Then either $\rho_{p,q}(N)$ is trivial or $\rho_{p,q}(\Gamma_k N)=\mathrm{PSL}_{d_{p,g}}(\mathbb{F}_q)$ for any $k\geq 1.$
	\end{enumerate}	
\end{lem}
\begin{proof}
	\begin{enumerate}[(a)]
		\item By \cite{LW05}, the image of $\rho_p( Mod(\Sigma_g))$ is dense in $\mathrm{PSU}_{d_{p,g}}$ for any $g\geq 2$ and any prime $p\geq 5.$ 
		
		Recall that $\mathrm{PSU}_d$ is a simple group for any $d>1,$ and moreover, that $d_{g,p}>1$ for any $g\geq 2$ and $p\geq 5.$ 
		
		Because $\overline{\rho_p(N)}\triangleleft \overline{\rho_p(Mod(\Sigma))}$, this implies that $\overline{\rho_p(N)}$ is either trivial or $\mathrm{PSU}_{d_{p,g}}$.
		Now, $\mathrm{PSU}_{d_{p,g}}$ is also a perfect group, therefore $$\overline{\rho_p(\Gamma_k N)}=\overline{\Gamma_k \rho_p(N)}=\Gamma_k \overline{\rho_p(N)}=\Gamma_k PSU_{d_{p,g}}=PSU_{d_{p,g}}.$$

		\item The proof is exactly the same as the previous case, using that $\mathrm{PSL}_{d_{p,g}}(\mathbb{F}_q)$ is simple and perfect.
	\end{enumerate}
\end{proof}
\begin{lem}\label{lem:nonTrivialTorelli}
	Let $g\geq 2,$ let $p\geq 5$ be a prime and let $q=1 \pmod{p}$ be a prime. Then $\rho_{p,q}(I(\Sigma_g))$ is not trivial.
\end{lem}
\begin{proof}
	Let $\gamma$ be a separating curve in $\Sigma_g$ and let $t_{\gamma}$ be the associated Dehn twist. The spectrum of $\rho_p(t_{\gamma})$ consists of $(-\zeta_p)^{i^2-1}$ for $i=1,\ldots,\frac{p-1}{2}.$ Note that $u_p=-\zeta_p$ is a primitive $p$-th root of unity. When reducing mod $J$ where $\Z[p^{-1},\zeta_p]/J\simeq \mathbb{F}_q,$ the elements $u_p^i$ for $i\in \Z/p\Z$ are all different mod $J:$ otherwise, one would get that $u_p^i-u_p^j\in J,$ but $u_p^i-u_p^j$ is a unit in $\Z[p^{-1},\zeta_p]$, since $p=\prod_{i=1}^{p-1}(1-u_p^i)$.
	
	Hence $\rho_{p,q}(t_{\gamma})\neq I_{d_{p,g}} \pmod{ J},$ and since $t_{\gamma}\in I(\Gamma)$ (in fact, it even belongs to the Johnson kernel of $\Sigma$), we have that $\rho_{p,q}(I(\Sigma_g))$ is not trivial.
\end{proof}

\begin{prop}
	\label{prop:densityGammakI} Let $g\geq 2,$ let $p\geq 5$ be a prime and $q=1 \pmod{ p}$ be another prime. Also let $k\geq 1,$ let $n$ be coprime to $p$ and let $H$ be a finite index normal subgroup of $\mathrm{Mod}(\Sigma),$ and write
	$$N=\Gamma_k I(\Sigma_g) \cap T_n(\Sigma_g)\cap H.$$
	Then $\rho_p(N)$ is dense in $PSU_{d_{p,g}}.$ Furthermore, if $q$ is such that $\rho_{p,q}(\mathrm{Mod}(\Sigma))=\mathrm{PSL}_{d_{p,g}}(\mathbb{F}_q),$ and the index of $H$ is strictly less than the order of $\mathrm{PSL}_{d_{p,g}}(\mathbb{F}_q)$ then 
	$$\rho_{p,q}(N)=\mathrm{PSL}_{d_{p,g}}(\mathbb{F}_q)$$
	also. 
\end{prop}
\begin{proof}
	Write $un+vp=1$ for some integer $u,v\in \Z.$
	First, notice that $\rho_p$ or $\rho_{p,q}$ are non-trivial on $I(\Sigma_g)\cap T_n(\Sigma_g):$ indeed, for $\gamma$ a non-trivial separating closed curve it suffices to consider
	$$\rho_{p,q}(t_{\gamma})=\rho_{p,q}(t_{\gamma}^{un})\in I(\Sigma_g)\cap T_n,$$
	where the equality follows from $un+vp=1$ and the fact that $\rho_{p}(t^p_{\gamma})=I_{d_{p,g}}$ for any Dehn twist $t_{\gamma}.$
	
	Lemma \ref{lem:nonTrivialTorelli} implies that this element is non trivial, and thus Lemma \ref{lem:SimplePerfect} immediately gives the proposition for the subgroup $N_0=\Gamma_kI(\Sigma_g)\cap T_n.$ To get it for $N,$ it suffices to remark that a finite index subgroup of a dense subgroup of $PSU_{d_{p,g}}$ is again dense in $PSU_{d_{p,g}},$ and that a normal subgroup of $\mathrm{PSL}_{d_{p,g}}(\mathbb{F}_q)$ of index strictly less than its order is all of $\mathrm{PSL}_{d_{p,g}}(\mathbb{F}_q).$
\end{proof}

\section{The main results}\label{sec:proofs}

We can now prove a more precise version of Theorem \ref{thm:intro1}.
\begin{thm}\label{thm:main}
	Take $k\geq 1$, $p\geq 5$ prime, $n$ coprime to $p$, $\mathcal{G}=\{G_1,\dots, G_r\}$ a finite collection of finite groups, $M$ a compact oriented $p$-good $3$-manifold, $g$ a positive integer and $v>0$. Let $N$ be any compact oriented manifold; then there is a hyperbolic $3$-manifold $N'$ with the following properties:
	
	\begin{itemize}
		\item  $N'$ is $(Y_k,T_n,\mathcal{G})$-equivalent to $N$;
		\item $N'$ does not contain any essential, non-boundary parallel surface of genus less than $g$;
		\item $N'$ has volume larger than $v$;
		\item $I_p(M)\nsubseteq I_p(N')$ which implies that $N'$ does not embed in $M$.
	\end{itemize}
\end{thm}

Theorem \ref{thm:intro1} follows from Theorem \ref{thm:main} by applying Corollary \ref{cor:convergence}.
\begin{proof}
	Let $\Sigma$ be a Heegaard surface for $N$ (after stabilizing we can assume its genus is $2$ or more). Let $H$ be the handlebody of the Heegaard decomposition and $C$ the compression body; its lower boundary is $\Sigma':=\partial N$ and its upper boundary is $\Sigma$. 	
	
	Consider a $\phi\in Mod(\Sigma)$ and $N_\phi$ the manifold obtained by a Torelli twist of $\phi$; we have that $RT_p(N_\phi)=RT_p(C)\circ \rho_p(\phi)(RT_p(H))$, where we consider $C$ as cobordism from $\Sigma$ to $\Sigma'$. 
	
	Notice that since $\Sigma'$ is obtained by compressing $\Sigma$ along a non-empty collection of disjoint essential simple closed curves, we have $\dim RT_p(\Sigma')<\dim RT_p(\Sigma)$ and thus $\ker(RT_p(C))\pmod{J}$ is not trivial. Notice furthermore that $RT_p(H)\neq 0 \pmod{J}$. If $\phi$ is such that 
	\begin{equation}\label{eq:zeroModJ}\rho_p(\phi)(RT_p(H)) \subset \ker RT_p(C) \pmod{ J}\end{equation} for some ideal $J\subseteq \Z\left[p^{-1},\zeta_p\right]$, then $RT_p(N_\phi)=0 \pmod{ J}$, which implies that $I_p(N_\phi)\subseteq J$.
	
	By Theorem \ref{thm:MasbaumReid}, there are infinitely many $q$'s such that $\rho_{p,q}(Mod(\Sigma))=\mathrm{PSL}_{d_{p,g}}(\mathbb{F}_q)$; there are also infinitely many $q$'s such that the order of the latter group is bigger than the index of $\cap_i \ker \rho^{G_i}$, since it increases with $q$. Therefore by Proposition \ref{prop:densityGammakI} we have that $\rho_{p,q}(\Gamma_kI(\Sigma)\cap T_n\cap_i \ker\rho^{G_i})=\mathrm{PSL}_{d_{p,g}}(\mathbb{F}_{q})$ for infinitely many $q$'s. This means that for infinitely many  maximal ideals $J\subseteq \Z[p^{-1},\zeta_p]$), we can find $\phi\in \Gamma_kI(\Sigma)\cap T_n\cap_i \ker\rho^{G_i}$ that satisfies Equation \ref{eq:zeroModJ} and thus such that $I_p(N_\phi)\subseteq J\subseteq \Z[p^{-1},\zeta_p]$.
	
	Since $M$ is $p$-good, we know that $I_p(M)\neq 0$; therefore $I_p(M)\nsubseteq J$ for all but finitely many maximal ideals $J$. We can then choose a $J$ such that $I_p(M)\nsubseteq J$ and that there is $\phi$ with the required properties such that $I_p(N_\phi)\subseteq J$; this implies that $N_\phi$ cannot embed in $M$. 
	
	Therefore we have a manifold satisfying the requirements of the theorem except for hyperbolicity. If $N_\phi$ is not hyperbolic, then by Proposition \ref{prop:hyperbolic} and the subsequent remark, there exists hyperbolic $N'$ that is $(Y_k, T_{np},\mathcal{G})$-equivalent to $N_\phi$. In particular, $N'$ is $(Y_k,T_n,\mathcal{G})$-equivalent to $N_\phi$ and thus to $N$; because of Lemma \ref{lem:Tnequiv}, since $N'$ is $T_p$-equivalent to $N_\phi$, $I_p(N')=I_p(N_\phi)$ and thus $N'$ does not embed in $M$ either.

\end{proof}

\begin{cor}\label{cor:pairwiseNotEmbedding}
	Let $M$ be a compact oriented $3$-manifold, any $k,n$ be natural numbers.
	\begin{enumerate}
		\item If $M$ is very good, then for any $l\geq 1$ there exist $M_1,\dots, M_l$ connected compact hyperbolic manifolds which are all $(Y_k,T_n)$-equivalent to $M,$ but such that for any $i\neq j$, $M_i$ does not embed in $M_j.$
		\item Assume that $M$ embeds in a rational homology sphere, then there is an infinite sequence $(M_i)_{i\geq 1}$ of compact oriented hyperbolic $3$-manifolds, all $(Y_k,T_n)$-equivalent to $M,$ but such that for any $i\neq j,$ the $3$-manifold $M_i$ does not embed in $M_j.$
	\end{enumerate}

\end{cor}
Note that being $(Y_k,T_n)$-equivalent in particular implies having homeomorphic boundaries; hence the corollary is mostly interesting if $M$ has non-empty boundary.
\begin{proof}
	(1) Let $p_1,\dots, p_l$ be primes that do not divide $n$ and such that $M$ is $p_i$-good, and call their product $d$.  By Theorem \ref{thm:main}, for each $i$ there is a hyperbolic $M_i$ which is $(Y_k,T_{nd/p_i})$-equivalent to $M$ but such that $I_{p_i}(M)\nsubseteq I_{p_i}(M_i)$. However, since $M_i$ is $T_{p_j}$-equivalent to $M$, $I_{p_j}(M)=I_{p_j}(M_i)$ for all $i\neq j$, which implies that $I_{p_i}(M_i)\nsubseteq I_{p_i}(M_j)$ and therefore $M_i$ does not embed in $M_j$.
	
	(2) We first note that any $3$-manifold which is $Y_1$-equivalent to a subspace of rational homology sphere is itself a subspace of a rational homology sphere. We construct the sequence $(M_i)_{i\geq 1},$ as well as a sequence of prime numbers $(p_i)_{i\geq 1}$ coprime to $n$ inductively, so that 
	\[\forall j<i, \ I_{p_j}(M_i)=I_{p_j}(M)\]
	and
	\[ I_{p_i}(M_1)I_{p_i}(M_2)\ldots I_{p_i}(M_{i-1})I_{p_i}(M) \nsubseteq I_{p_i}(M_i).\]
	Note that these conditions immediately imply that $M_i$ does not embed in $M_j$ for $i\neq j.$

	To start with, since $M$ is very good by Proposition \ref{prop:VeryGood}, let $p_1\geq 5$ be a prime such that $M$ is $p_1$-good. By Theorem \ref{thm:main}, there exists a hyperbolic $M_1$ which is $(Y_k,T_n)$-equivalent to $M,$ such that $I_{p_1}(M)\nsubseteq I_{p_1}(M_1).$
	
	 Assume that $M_1,\ldots, M_l$ have been constructed. Since they are all $(Y_k,T_n)$-equivalent to $M,$ they are all $Y_1$-equivalent to $M,$ hence they are all subspaces of rational homology spheres, and by Proposition \ref{prop:VeryGood}, they are all very good.
	
	Let $p_{l+1}\geq 5$ be a prime which does not divide $N=np_1\ldots p_l,$ and such that $RT_{p_{l+1}}(M)\neq 0$ and that $RT_{p_{l+1}}(M_i)\neq 0$ for any $i\leq l.$ Then the ideal 
	$$J=I_{p_{l+1}}(M_1)I_{p_{l+1}}(M_2)\ldots I_{p_{l+1}}(M_{l})I_{p_{l+1}}(M)$$ of $\Z[\zeta_{p_{l+1}},\frac{1}{p_{l+1}}]$ is non-zero. By the proof of Theorem \ref{thm:main}, there exists a hyperbolic $3$-manifold $M_{l+1}$ that is $(Y_k,T_N)$-equivalent to $M,$ and such that
	$$J\nsubseteq I_{p_{l+1}}(M_{l+1}).$$
	
	This completes the induction step; we have constructed an infinite sequence of $3$-manifolds that are all $(Y_k,T_n)$-equivalent to $M$ but pairwise do not embed in each other.
\end{proof}

\begin{rem}
	The classical obstructions to embeddings listed in Section \ref{sec:overview} (namely the ones coming from homology and Gromov norm) are inequalities; this means that if we can use them to obstruct the embedding of $M$ into $N$, we necessarily cannot use them to also obstruct the embedding of $N$ into $M$, and thus they cannot be used to prove anything similar to Corollary \ref{cor:pairwiseNotEmbedding}.
\end{rem}
\section{Lower bounds for non-embedding probabilities}
\label{sec:proba}
In this section, we will give a quantitative version of Theorem \ref{thm:main}, using the variation of the Dunfield-Thurston model of random $3$-manifolds that was described in the introduction.

We briefly recall the setting. Let $N$ be a compact oriented $3$-manifold, and let $N=C\cup_\phi H$ be a Heegaard splitting of $N$ with Heegaard surface $F$ of genus $g\geq 2.$ Fix $k,n\geq 1,$ and let $G=\Gamma_k I(F)\cap T_n.$ Fix also $S$ a generating set for $G$ (note that $S$ is not necessarily finite; in fact $G$ may not be finitely generated), and $\mu$ a probability measure on $S$ with full support.

We let $(f_d)_{d\geq 0}\in G$ be a sequence of random variables given by
$$f_d=\varphi_1\ldots \varphi_d$$
where the $\varphi_i$ are independent identically distributed random variables with law $\mu.$

We will set $d_p(F):=\dim RT_p(F)$ for any compact oriented surface $F$ and any prime number $p\geq 5.$
Finally, let $M$ be a $p$-good $3$-manifold, for some prime $p$ coprime to $n,$ and let $J\subset \Z[p^{-1},\zeta_p]$ be an ideal such that
\begin{enumerate}
	\item $\Z[p^{-1},\zeta_p]/J\simeq \mathbb{F}_q$
	\item The quantum representation $\rho_{p,q}$ of $\mathrm{Mod}(F)$ surjects onto $\mathrm{PSL}_{d_p(F)}(\mathbb{F}_q)$
	\item $I_p(M)\nsubseteq J.$
\end{enumerate}

 The following is a quantitative version of Theorem \ref{thm:main2} of the intro.
\begin{thm}
	\label{thm:proba} Let $$N'_d=C\cup_{\phi\circ f_d}H$$ be the random $3$-manifold described above. Then 
	$$\underset{d\rightarrow\infty}{\liminf} \ \mathbb{P}(N'_d \ \textrm{does not embed in } \ M)\geq \frac{q^{d_p(F)-d_p(\partial N)}-1}{q^{d_p(F)}-1}.$$
\end{thm}
\begin{rem}\label{rk:lowerBound}
	When $p=5$ and $\partial N=T^2$ for example, by choosing a Heegaard splitting of $N$ of large genus, this lower bound can be made arbitrarily close to $\frac{1}{q^2},$ since $d_5(T^2)=2$ and $d_5(\Sigma_g)\underset{g\rightarrow \infty}{\longrightarrow}+\infty.$
\end{rem}

We break the proof of Theorem \ref{thm:proba} into two lemmas:

\begin{lem}
	\label{lem:unifDistrib} Let $k,n\geq 1.$ For $(f_d)_{d\geq 0}$ a random walk on $\Gamma_kI(F)\cap T_n,$ the distribution of $\rho_{p,q}(f_d)$ tends to the uniform distribution on $\mathrm{PSL}_d(\mathbb{F}_q).$
\end{lem}
\begin{proof}
	We note that $\rho_{p,q}(f_d)=\rho_{p,q}(\varphi_1)\ldots \rho_{p,q}(\varphi_d).$ The $\rho_{p,q}(\varphi_i)$ are independent identically distributed random variables in $\mathrm{PSL}_{d_p(F)}(\mathbb{F}_q),$ with law $\nu:=\left(\rho_{p,q}\right)_*(\mu).$ The law $\nu$ has support $\rho_{p,q}(S),$ which is a finite generating set of $\mathrm{PSL}_{d_p(F)}(\mathbb{F}_q),$ since $\rho_{p,q}$ restricted to $\Gamma_kI(F)\cap T_n$ is surjective by Proposition \ref{prop:densityGammakI}. Then $\rho_{p,q}(f_d)$ is an irreducible and aperiodic Markov chain on $\mathrm{PSL}_{d_p(F)}(\mathbb{F}_q);$ where irreducibility follows from $\rho_{p,q}(S)$ being a generating set and aperiodicity follows from $\mathrm{PSL}_{d_p(F)}(\mathbb{F}_q)$ being a simple group. By the theory of Markov chains (see for example \cite[Theorem 4.9]{DPW09}), we deduce that $\rho_{p,q}(f_d)$ converges to the uniform distribution on $\mathrm{PSL}_{d_p(F)}(\mathbb{F}_q).$
\end{proof}

\begin{lem}
	\label{lemma:probaKernel} Let $n\geq 2$ and let $V$ be a subspace of $(\mathbb{F}_q)^n$ of dimension $m.$ Let $X$ be a uniformly distributed random variable in $\mathrm{PSL}_{n}(\mathbb{F}_q).$ Then for any non-zero vector $v\in (\mathbb{F}_q)^n,$
	$$\mathbb{P}(Xv\in V)=\frac{q^m-1}{q^n-1}.$$ 
\end{lem}
\begin{proof}
	Because $X\in \mathrm{PSL}_n(\mathbb{F}_q)$, $Xv$ is a uniformly distributed random point in $\mathbb{P}(\mathbb{F}_q^n)$. Since there are $\frac{q^m-1}{q-1}$ points in $\mathbb{P}(V)$ and $\frac{q^n-1}{q-1}$ points in $\mathbb{P}(\mathbb{F}_q^n)$, the formula follows.
\end{proof}
\begin{proof}[Proof of Theorem \ref{thm:proba}]
	By Lemma \ref{lem:unifDistrib}, $\rho_{p,q}(\phi \circ f_d)$ tends to the uniform distribution in $\mathrm{PSL}_{d_p(F)}(\mathbb{F}_q).$ 
	
	We apply Lemma \ref{lemma:probaKernel} with $n=d_p(F),$ $v=RT_p(H)$ and $V=\ker RT_p(C) \pmod{J}.$ Note that since $RT_p(C):RT_p(F)\longrightarrow RT_p(\partial N),$ its kernel is at least of dimension $d_p(F)-d_p(\partial N).$
	
	Since if $\rho_p(\phi\circ f_d)(RT_p(H)) \in \ker RT_p(C) \pmod{J},$ then $I_p(N'_d)\in J,$ and since $I_p(M)\nsubseteq J,$ we conclude that 
	$$\underset{d\rightarrow\infty}{\liminf} \ \mathbb{P}(N'_d \ \textrm{does not embed in } \ M)\geq \frac{q^{d_p(F)-d_p(\partial N)}-1}{q^{d_p(F)}-1},$$
	as claimed.
\end{proof}

\begin{rem}
	A random $3$-manifold is hyperbolic, that is to say, the probability of obtaining a hyperbolic $3$-manifold from a random walk of length $d$ in the mapping class group has limit $1$ as $d\ra\infty$. This appeared in \cite{Mah} for the closed case; the case with boundary follows for example from the reasoning at the start of Proposition \ref{prop:hyperbolic} plus the work of \cite{Mah}. Similarly, a random $3$-manifold contains no essential, non-boundary parallel surfaces of genus less than a given $g$ by \cite{Sch} and \cite{Mah}. The volume of a hyperbolic closed $3$-manifold obtained from a random walk of length $d$ grows linearly in $d$ almost surely, by \cite{Via}; a modification of the techniques in that paper should lead to an analogous result for manifolds with boundary.
\end{rem}
\bibliographystyle{alpha}
\bibliography{biblio}
\end{document}